\documentclass[10pt,twoside,a4paper]{amsart}

\usepackage[english]{babel}
\usepackage[utf8]{inputenc}

\usepackage{amsfonts,amsthm,amssymb,amsmath,amsthm,latexsym,wasysym,stmaryrd,mathrsfs,dsfont,txfonts}
\usepackage{accents,multirow,rotating,subfigure,graphicx,bigstrut,lscape,multicol,fancyhdr,enumerate,accents,mathtools,exscale,picins,afterpage}

\usepackage[normalem]{ulem}

\usepackage{pdftricks}
\usepackage{color}
\usepackage[pdftex,all]{xy}

\usepackage{hyperref}
\usepackage{appendix}
\usepackage[justification=centering]{caption}

\graphicspath{{Figures/}}

\theoremstyle{theorem}
\newtheorem {theo}{Theorem}[section]
\newtheorem*{theo*}{Main Theorem}
\newtheorem {lemme}[theo]{Lemma}
\newtheorem*{lemme*}{Lemma}
\newtheorem {prop}[theo]{Proposition}
\newtheorem*{prop*}{Proposition}

\newtheorem*{cor*}{Corollary}

\newtheorem*{cor_proof*}{Corollary (of the proof)}

\newtheorem*{conjecture*}{Conjecture}
\theoremstyle{definition}
\newtheorem {defi}[theo]{Definition}
\newtheorem*{defi*}{Definition}

\newtheorem*{nota*}{Notation}
\theoremstyle{remark}
\newtheorem {remarque}[theo]{Remark}
\newtheorem*{remarque*}{Remark}

\newtheorem*{warning*}{Warning}

\newtheorem*{remarques*}{Remarks}

\newtheorem*{warnings*}{Warnings}

\newtheorem*{convention*}{Convention}

\newtheorem*{exemple*}{Example}

\newtheorem*{exemples*}{Examples}

\newtheorem*{question*}{Question}

\newtheorem*{questions*}{Questions}

\newtheorem*{fact*}{Fact}
\newtheorem*{acknowledgments}{Acknowledgments}

\def\LL{{\mathcal L}}

\def\N{{\mathds N}}

\def\PP{\mathcal{P}}

\def\R{{\mathds R}}

\def\Z{{\mathds Z}}

\def\2Z{{\fract{\Z}{2\Z}}}

\def\e{\varepsilon}

\newcommand{\fract}[2]{\hbox{\leavevmode
  \kern.1em \raise .25ex \hbox{\the\scriptfont0 $#1$}\kern-.1em }\big/
  {\hbox{\kern-.15em \lower .5ex \hbox{\the\scriptfont0 $#2$}} }}

\newcommand{\dessin}[2]{
  \vcenter{\hbox{\includegraphics[height=#1]{#2.pdf}}}}

\def\comp{\textnormal{c}}

\def\nF{\textnormal F}

\def\nR{\textnormal R}

\def\wGraph{w\mathcal{L}}
\def\wGraphL{\wGraph_n}

\def\wGraphLh{\wGraphL ^{\textnormal{sv}}}

\DeclareMathOperator{\Spun}{Spun}
\DeclareMathOperator{\Tube}{Tube}

\def\SV{\textnormal{SV}}
\def\TaH{\textnormal{TaH}}

\def\sv{\textnormal{sv}}

\begin{document} 

\title{Characterization of the reduced peripheral system of links} 
\author[B. Audoux]{Benjamin Audoux}
         \address{Aix Marseille Univ, CNRS, Centrale Marseille, I2M, Marseille, France}
         \email{benjamin.audoux@univ-amu.fr}
\author[J.B. Meilhan]{Jean-Baptiste Meilhan} 
\address{Univ. Grenoble Alpes, CNRS, Institut Fourier, F-38000 Grenoble, France}
	 \email{jean-baptiste.meilhan@univ-grenoble-alpes.fr}
%
%\subjclass[2000]{57M25, 57M27}
%
\begin{abstract} 
The reduced peripheral system was introduced by Milnor in the fifties
for the study of links up to link-homotopy, {\it i.e.} 
up to isotopies and crossing changes within each link component.
However, for four or more components, this invariant does not yield a complete
link-homotopy invariant.
This paper provides two characterizations of links having the
  same reduced peripheral system: a diagrammatical one, in terms of link
  diagrams, seen as welded diagrams up to self-virtualization, and a 
  topological one, in terms of ribbon solid tori in 4--space up to ribbon link-homotopy. 
\end{abstract} 

\maketitle

\section*{Introduction}
A celebrated consequence of Waldhausen's theorem on Haken $3$--manifolds \cite{Waldo}
is that the fundamental group of the complement, endowed with a peripheral system, forms a complete invariant of links in the 3-sphere up to ambient isotopy. 
The peripheral system is the data, for each component of a link $L$ in $S^3$, of a pair of elements $\{m_i,l_i\}$ of $\pi_1(S^3\setminus L)$---a 
meridian and a preferred longitude---that generates the fundamental group of the corresponding boundary component of $S^3\setminus L$. 
Although rather intractable in practice, the peripheral system is 
nonetheless a fundamental link invariant, 
and it is natural to expect that some weaker equivalence relations than ambient isotopy could be classified by an appropriate adaptation of the  peripheral system. 
During the fifties, this 
has been the strategy of J.~Milnor in his attempt to classify links up to \emph{link-homotopy} \cite{Milnor}, that is up to homotopy deformations during which distinct connected components remain disjoint at all time.
In order to address 
this link-homotopy classification problem, 
Milnor introduced the \emph{reduced peripheral system}.  
Roughly speaking, the \emph{reduced fundamental group} $R\pi_1(L^\comp)$ of a link $L$ is the largest quotient of the fundamental group of the complement 
where any generator commutes with any of its conjugates; 
if $\{\mu_i,\lambda_i\}_i$ 
is a peripheral system for $L$, with image $\{m_i,l_i\}_i$ under the projection onto $R\pi_1(L^\comp)$, 
then a \emph{reduced peripheral system} of $L$ is $\{m_i,l_i N_i\}_i$, where $N_i$ is the normal subgroup of $R\pi_1(L^\comp)$ generated by $m_i$. 

The reduced peripheral system, however, only yields a complete link-homotopy invariant for links with at most $3$ components. 
The $4$--component case was tackled by J.~Levine \cite{Levine} only 40 years later, using a smaller normal subgroup for defining the reduced longitudes.
As a matter of fact, there exists a pair of $4$--component links, exhibited by J.R.~Hughes, with equivalent reduced peripheral systems but which are link-homotopically distinct \cite{Hughes93}. 
It seems still unknown whether Levine's peripheral system classifies links up to link-homotopy. 
In fact, this classification was achieved by N.~Habegger and X.S.~Lin by a rather different approach, which relies on representing links as the closure of string links \cite{HL}. 

In the present paper, we take advantage of the recently developped theory
of welded links to provide, in an elementary way, a four-dimensional topological characterization of the information captured by Milnor's reduced peripheral system:
\begin{theo*}\label{thm:C}
Let $L$ and $L'$ be two oriented links in the $3$-sphere. The following are equivalent: 
 \begin{itemize}
 \item[i.] $L$ and $L'$ have equivalent reduced peripheral systems;  
 \item[ii.] $L$ and $L'$ are \textrm{sv}-equivalent, as welded links; 
 \item[iii.] $\Spun^\bullet(L)$ and $\Spun^\bullet(L')$ are ribbon link-homotopic, as ribbon immersed solid tori.
 \end{itemize}
\end{theo*}

Here, \emph{welded links} are generalized link diagrams, where we allow for virtual crossings in addition to the usual crossings, regarded up to an extended set of Reidemeister moves. 
This is a sensible generalization in the sense that classical links inject into welded links, and that the fundamental group and (reduced) peripheral system naturally extends to this larger class of objects. 
Part ii then gives a diagrammatic characterization of the reduced peripheral systems of links, by regarding them as welded links via their diagrams, up to \emph{$\sv$--equivalence}, 
which is the equivalence relation  generated by the replacement of a classical  crossing involving two strands of a same component by a virtual one; 
this stresses in particular the fact that the $\sv$--equivalence is a
refinement of link-homotopy for classical links, see Remark \ref{rem:scc}. 
This characterization actually follows from a more general result, which classifies all welded links up to $\sv$--equivalence, see Theorem \ref{thm:weldedclassif}. 
Also, it follows that the non link-homotopic $4$--component links exhibited by Hughes in \cite{Hughes93} are $\sv$-equivalent; 
this is made explicit in Appendix \ref{onvamangerdeschips}.  

Part iii gives a topological characterization, in terms of $4$--dimensional topology. 
A classical construction dating back to Artin \cite{cEmilletueur} produces a knotted surface in $4$--space from a link in $3$--space by spinning it around some plane. 
By spinning as well projection rays from the link to the plane, this can be extended to a map $\Spun^\bullet$ producing \emph{ribbon-immersed solid tori}, 
\emph{i.e.} solid tori in $4$--space intersecting along only ribbon singularities. The \emph{ribbon link-homotopy} for such objects is 
a notion of link-homotopy within the realm of ribbon-immersed solid tori, which allows for the removal/insertion of such ribbon singularities 
inside the same connected component.
The reduced peripheral system for links hence appears in this way as an intrinsically $4$--dimensional invariant, rather than a $3$-dimensional one.\footnote{One might expect for this $4$-dimensional incarnation to be
in terms of knotted surfaces, we explain in Section
\ref{sec:surfaces} why this is not the case.}
As above, this characterization is obtained as a consequence of a more general result, characterizing the reduced peripheral system of welded links in terms of  $4$--dimensional topology; see Theorem \ref{thm:4-dimclassif}. 
It is thus noteworthy that our purely topological characterization $\textrm{i}\Leftrightarrow \textrm{iii}$ for classical links is actually 
obtained as an application of virtual/welded knot theory. 

\begin{acknowledgments}
The authors thank A.~Yasuhara for useful comments on an earlier version of this paper. 
\end{acknowledgments}

\section{The reduced peripheral system of classical and welded links}

\subsection{Welded links} \label{sec:combinatorics}
In this section, we review the theory of welded links and Gauss diagrams. 

\begin{defi} 
  \begin{itemize}
  \item[]
  \item An $n$-component \emph{welded diagram} is a planar immersion
    of $n$ ordered and oriented circles, whose singular set is a
    finite number of transverse double points, each double point being
    labelled 
    either as a \emph{positive or negative (classical) crossing}, or as a \emph{virtual
      crossing}:
    \[
    \begin{array}{ccccc}
      \dessin{1cm}{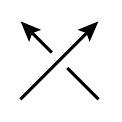}&\hspace{1cm}&\dessin{1cm}{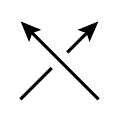}&\hspace{1cm}&\dessin{1cm}{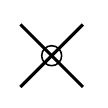}\\
      \textrm{positive crossing}&&\textrm{negative crossing}&&\textrm{virtual crossing}
    \end{array}.
    \]
  \item  We denote by $\wGraphL$ the set of $n$-component welded
    diagrams up to the following \emph{welded moves}:
\[
  \begin{array}{c}
    \begin{array}{ccc}
   \dessin{1.5cm}{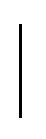}
    \stackrel{\textrm{vR1}}{\longleftrightarrow}  \dessin{1.5cm}{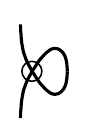} \, & \, 
    \dessin{1.5cm}{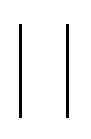} \stackrel{\textrm{vR2}}{\longleftrightarrow}  \dessin{1.5cm}{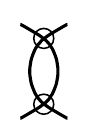}
   \, & \, 
   \dessin{1.5cm}{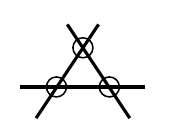} \stackrel{\textrm{vR3}}{\longleftrightarrow}  \dessin{1.5cm}{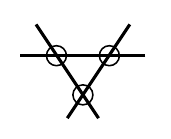}
   \end{array}
   \\
    \textrm{virtual Reidemeister moves}\\[0.4cm]
    \begin{array}{cc}
      \dessin{1.5cm}{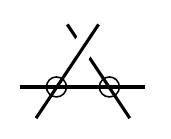}\stackrel{\textrm{Mixed}}{\longleftrightarrow} \dessin{1.5cm}{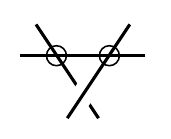} \, & \, 
      \dessin{1.5cm}{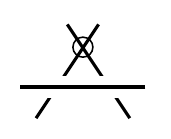}\ \stackrel{\textrm{OC}}{\longleftrightarrow}
      \dessin{1.5cm}{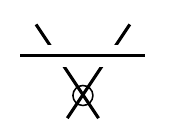}
    \end{array}
    \\
    \textrm{Mixed and OC\footnotemark moves}
  \end{array}
\]
\footnotetext{Here, OC stands for \emph{over-commute}, as
    a strand is passing over a virtual crossing; note that 
    the corresponding \emph{under-commute} move is forbidden.}
and classical Reidemeister moves R1, R2 and R3, which are the three usual moves of classical knot theory. 
 Elements of $\wGraphL$ are called \emph{welded links}. 
  \end{itemize}
\end{defi}

A welded diagram with no virtual crossing is called \emph{classical}. 
It is well-known that this set-theoretical inclusion induces an injection of the set $\mathcal{L}_n$ of $n$-component classical link diagrams up to classical Reidemeister moves, 
into $\wGraphL$;
as pointed out in  Remark \ref{rem:inject}, this follows from the fact that the peripheral system is a complete link invariant. 

\medskip

An alternative approach to welded links, which is often more tractable in practice, is through the notion of Gauss diagrams.
  
\begin{defi}
An $n$-component \emph{Gauss diagram} is an abstract collection of $n$ ordered and oriented circles, together with disjoint signed arrows whose endpoints are pairwise disjoint points of these circles. 
For each arrow, the two endpoints are called \emph{head} and \emph{tail},
with the obvious convention that the arrow orientation goes from the tail to the head. 
\end{defi}

To a welded diagram corresponds a unique Gauss diagram, given by joining the two preimages of each classical crossing by an arrow, 
oriented from the overpassing to the underpassing strand and labelled by the crossing sign. 

\begin{defi}
Two Gauss diagrams are \emph{welded equivalent} if they are related by
a sequence of the following \emph{welded moves}:
\[
    \begin{array}{cccc}
      \dessin{1.5cm}{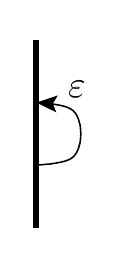} \stackrel{\textrm{R1}}{\longleftrightarrow}  \dessin{1.5cm}{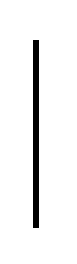} 
\, \,  \,  &
      \,  \,  \,   \dessin{1.5cm}{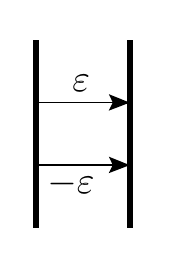} \stackrel{\textrm{R2}}{\longleftrightarrow}  \dessin{1.5cm}{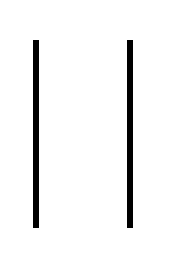} \,  \,  \, 
      & \,  \,  \, 
\dessin{1.5cm}{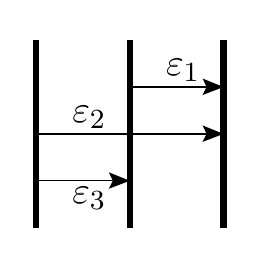} \stackrel{\textrm{R3}}{\longleftrightarrow}  \dessin{1.5cm}{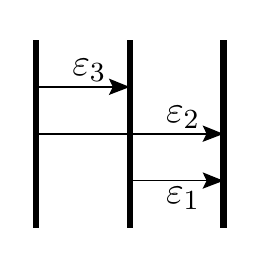} \,  \,  \, 
      & \,  \,  \, 
\dessin{1.5cm}{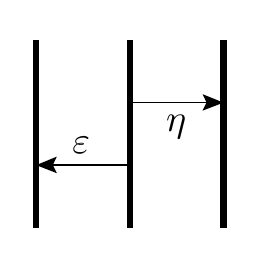} \stackrel{\textrm{OC}}{\longleftrightarrow}  \dessin{1.5cm}{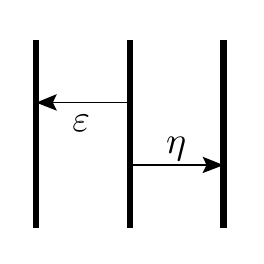},
    \end{array}
\]
where move R3 requires the additional sign condition that $\varepsilon_2\varepsilon_3=\tau_2\tau_3$, where
  $\tau_i=1$ if the $i^\textrm{th}$ strand (from left to right) is oriented upwards, and $-1$ otherwise.
\end{defi} 
As the notation suggests, these four moves are just the Gauss diagram analogues, using the above correspondance, 
of the three classical Reidemeister moves and the OC move for welded diagrams   
(the Gauss diagram versions of the virtual Reidemeister and Mixed
moves being trivial).
As a matter of fact, it is easily checked that
welded equivalence classes of Gauss diagrams are in one-to-one
correspondence with welded links. 

\begin{remarque}\label{rem:slide} 
 We will make use the following \emph{Slide} move
\[
    \begin{array}{c}
      \dessin{1.5cm}{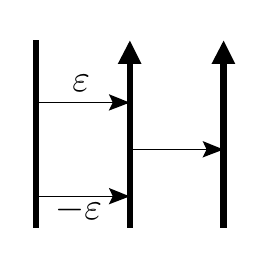} \stackrel{\textrm{Slide}}{\longleftrightarrow}  \dessin{1.5cm}{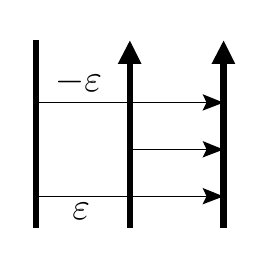}
    \end{array}
\hspace{2cm}
 \dessin{1.5cm}{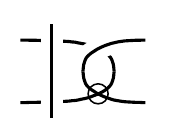} \stackrel{\textrm{Slide}}{\longleftrightarrow}  \dessin{1.5cm}{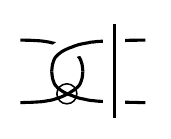}
\]
 \noindent which is easily seen to follow from the R2, R3 and OC moves.
Note that this is a Gauss diagram analogue of the Slide move on arrow diagrams \cite{arrow}. 
\end{remarque}
 
\medskip

In a welded diagram, a \emph{self-crossing} is a crossing where both preimages belong to the same component.   

\begin{defi}
  A \emph{self-virtualization} is a local move $\textrm{SV}$, illustrated
  below, which replaces a classical self-crossing by a virtual one. 
  The \emph{\textrm{sv}-equivalence} is the equivalence relation on welded diagrams generated by self-virtualizations. 
  We denote by $\wGraphLh$ the quotient of $\wGraphL$ under this relation.
     \[
   \begin{array}{ccc}
\dessin{1.2cm}{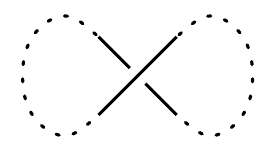}\stackrel{\textrm{SV}}{\longleftrightarrow}\dessin{1.2cm}{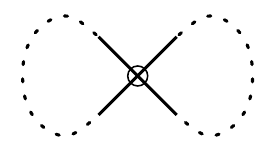}\stackrel{\textrm{SV}}{\longleftrightarrow}\dessin{1.2cm}{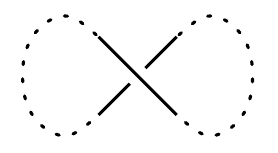} 
& \hspace{1cm}& 
\dessin{1cm}{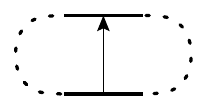}\stackrel{\textrm{SV}}{\longleftrightarrow}\dessin{1cm}{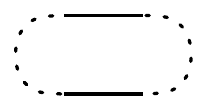}
    \end{array}
  \]
\end{defi}
At the Gauss diagram level, a self-crossing is represented by a \emph{self-arrow}, that is an arrow whose endpoints lie on the same component, and a self-virtualization move simply erases a self-arrow.

\begin{remarque}\label{rem:scc}
  The link-homotopy relation for classical links, as defined by Milnor, is generated by the self-crossing change, \emph{i.e.} the local move that exchanges the relative position of two strands of a same component. As the left-hand side of
  above picture suggests, a self-crossing change can be realized by two self-virtualizations, and the $\sv$--equivalence is thus a refinement of the link-homotopy relation for classical links. In \cite{ABMW}, self-virtualization was proven to actually extend classical link-homotopy for string links in the sense that two string-links are $\sv$--equivalent iff they are link-homotopic. Our main theorem, together with Hughes' counter-examples, shows that such an equivalence does not hold true for links.
\end{remarque}

We end this section with a normal form for Gauss diagrams up to self-virtualization.
\begin{defi}
  A Gauss diagram is \emph{sorted} if each circle
splits into two arcs, the $t$ and the $h$--arc,
 containing respectively tails only and heads only.
\end{defi}
\begin{remarque}\label{rem:SortedLabels}
  Up to OC moves, a sorted Gauss diagram $D$ is uniquely determined by
  the data of $n$ words
  $\Big\{\prod_{j=1}^{k_i}\mu_{s_{ij}}^{\varepsilon_{ij}}\Big\}_{\! i}$
  in the alphabet $\{\mu_1^{\pm1},\ldots,\mu_n^{\pm1}\}$, where the letter $\mu_{s_{ij}}^{\varepsilon_{ij}}$ 
  indicates that the $j\,$th head met on $C_i$ when running along its oriented $h$--arc is connected by an $\varepsilon_{ij}$--signed arrow to one of the tails on $C_{s_{ij}}$.
\end{remarque}

\begin{lemme}\label{lemma:Sorted}
Every Gauss diagram is $\sv$--equivalent to a sorted one.
\end{lemme}
\begin{figure}
  \[
  \TaH :  \dessin{1.2cm}{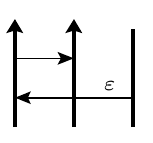}\ \xrightarrow[]{\textrm{R2}}\ \dessin{1.2cm}{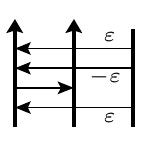}\ \xrightarrow[]{\textrm{Slide}}\ \dessin{1.2cm}{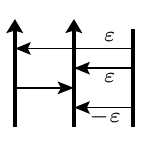}.
  \]  
  \caption{Tail across head move on Gauss diagrams}
  \label{fig:TaH}
\end{figure}
\begin{proof}
  Start with any given Gauss diagram, and choose any arbitrary order
  on the circles to sort them one by one as follows.
Using $\SV$ moves, remove first all self-arrows from the considered
circle, and then
gather all heads located on it using $\TaH$ (tail across head) moves, described in Figure \ref{fig:TaH}.
Since there is no self-arrow left on the considered circle, the two extra
arrows appearing in the latter moves won't have endpoints on the considered circle, and as their heads, resp. tails, are close to the already
existing arrow head, resp. tail, this won't unsort already sorted
components.
Note that these two extra arrows could happen to be self-arrows, but this does not conflict with the sorting procedure.
\end{proof}

\subsection{Welded link groups and peripheral systems}
\label{sec:wgroup}

Let $L$ be a welded diagram.

\begin{defi}
  \begin{itemize}
  \item[]
  \item The \emph{arcs} of $L$ are the maximal pieces of $L$ which do not underpass any classical crossing. 
An arc is hence either a whole component or a piece of strand which starts and ends at some (possibly the same) crossings; 
it might pass through some virtual crossings and overpass some
classical ones. At the level of Gauss diagrams, it corresponds to
portions of circles comprised between two heads.
\item The \emph{group of $L$}, denoted by $G(L)$, is defined by a Wirtinger-type presentation, where each arc yields a generator, 
and each classical crossing yields a relation, as follows: 
\[
  \dessin{1.8cm}{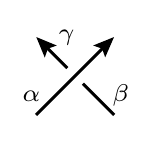}\ \leadsto\ \alpha^{-1}\beta\alpha\gamma^{-1}
  \hspace{1cm}
  \dessin{1.8cm}{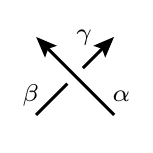}\ \leadsto\ \alpha\beta\alpha^{-1}\gamma^{-1}\\[-0.1cm]
\]
  \end{itemize}
\end{defi}

Since virtual crossings do not produce extra generator or relation, it is clear that virtual Reidemeister moves and Mixed moves preserve the group presentation. 
It is also easily checked that the isomorphism class of this group is invariant under classical Reidemeister and OC moves, 
and is thus an invariant of welded links \cite{Kauffman,Satoh}.
If $L$ is a diagram of a classical link $\LL$, then $G(L)$ is merely the fundamental group of the complement of an open tubular neighborhood of $\LL$ in $S^3$; 
in this case, an arc corresponds to the topological meridian which positively enlaces it.
By analogy, arcs of welded diagrams can be seen as some combinatorial meridians, and in what follows, 
we will often blur the distinction between arcs/meridians of $L$ and the corresponding generators of $G(L)$.
We will also regularly, and sometimes implicitly, make use of the simple fact that two meridians of a same components are always conjugate. 
 \medskip

\begin{defi}
  \begin{itemize}
  \item[]
  \item A \emph{basing} of $L$ is a choice 
        of one meridian for each component of $L$. 
  \item For each $i$, the \emph{$i\,$th preferred longitude} of
    $L$ with respect to the basing $\{\mu_1,\ldots,\mu_n\}$ is the
    element $\lambda_i\in G(L)$ obtained as
    follows: when running along the $i\,$th component of $L$, starting
    at the arc labelled by $\mu_i$ and following the orientation,
    write $\omega^\varepsilon$ when passing under an arc labelled by $\omega$ at
    a classical crossing of sign $\varepsilon$, and finally write
    $\mu_i^{-k}$, where $k$ is the algebraic number of classical self-crossings
    in the $i\,$th component.
  \item  A \emph{peripheral system} for $L$ is the group $G(L)$ endowed with the choice of a basing and
         the data, for each $i$, of the $i\,$th preferred longitude.
  \end{itemize}
\end{defi}

When $L$ is a classical link, a basing is the choice of a topological meridian for each component, and the $i\,$th preferred longitude represents a parallel copy of the $i\,$th component having linking number zero with it.  
Hence the above definitions naturally generalize the usual notion of
peripheral system of links.

\medskip

Two peripheral systems $\big(G, \{(\mu_i,\lambda_i)\}_i\big)$
and $\big(G, \{(\mu'_i,\lambda'_i)\}_i\big)$ are \emph{conjugate}
  if, for each $i$, there exists $\omega_i\in G$ such that $\mu'_i=\omega_i^{-1}\mu_i\omega_i$ and
  $\lambda'_i=\omega_i^{-1}\lambda_i\omega_i$. 
Two peripheral systems $\left(G, \{(\mu_i,\lambda_i)\}_i\right)$ and
$\left(G', \{(\mu'_i,\lambda'_i)\}_i\right)$ are \emph{equivalent} if
there exists an isomorphism $\psi: G'\rightarrow G$ such that $\big(G, \{(\mu_i,\lambda_i)\}_i\big)$
  and $\big(G, \{(\psi(\mu'_i),\psi(\lambda'_i))\}_i\big)$ are conjugate.

The following is well-known, see for example \cite[Prop.6]{Kim}. 
\begin{lemme}
Up to conjugation, peripheral systems are well-defined for welded
  diagrams and yield, up to equivalence, a well-defined invariant
  of welded links.
\end{lemme}
\begin{proof}
Suppose that $\{(\mu_i,\lambda_i)\}_i$ is a peripheral system of
the welded diagram $L$, and let $\mu'_i$ be another choice of meridian for the $i\,$th
  component, yielding hence another preferred $i\,$th longitude $\lambda'_i$. 
Then $\mu'_i=\omega_i^{-1}\mu_i\omega_i$ for some $\omega_i\in G(L)$, and by definition 
$\lambda'_i=\omega_i^{-1}(\lambda_i \mu_i^k) \omega_i {\mu'_i}^{-k}$. 
But substituting $\mu'_i$ for $\omega_i^{-1}\mu_i\omega_i$ in $\lambda'_i$ then gives 
$\lambda'_i=\omega_i^{-1}\lambda_i \mu_1^k \omega_i \omega_i^{-1} \mu_i^{-k} \omega_i = \omega_i^{-1}\lambda_i  \omega_i$. 
This proves that the peripheral system of $L$ is uniquely determined up to conjugation. 

Using this fact, it is then an easy exercise to check that equivalence classes of peripheral systems are well-defined for welded links, \emph{i.e.} that they are invariant under welded and classical Reidemeister moves. 
More precisely, by an appropriate choice of basing, one can check that each classical Reidemeister move induces an isomorphism of the groups of the diagrams which preserves each preferred longitude; the argument
is even simpler for welded Reidemeister moves.\\[-0.5cm]
{\parpic[r]{$\begin{array}{c}
   L\ \dessin{1.5cm}{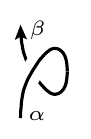} \stackrel{\textrm{R1}}{\longleftrightarrow}  \dessin{1.5cm}{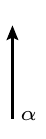}L'
   \end{array}$}

   As an elementary illustration, let us consider two welded diagrams $L$ and $L'$ which
  differ by an R1 move as shown on the right.
The generators $\alpha$ and $\beta$ of $G(L)$ shown in the figure satisfy $\alpha=\beta$, so $G(L)$ and $G(L')$ are clearly isomorphic. 
Pick $\beta\in G(L)$, resp. $\alpha\in G(L')$, as meridian for the depicted component of $L$, resp. $L'$. 

}
\noindent Then the corresponding preferred longitude of $L$ is of the form
$\omega\alpha\alpha^{-k}$ for some $\omega\in G(L)$ and some $k\in \mathbb{Z}$, while the
corresponding preferred longitude of $L'$ reads $\omega\alpha^{-k+1}$, since $L'$ contains one less positive self-crossing. 
Hence the above isomorphism from $G(L)$ to $G(L')$ preserves the peripheral system.
\end{proof}

\begin{remarque}\label{rem:inject}
The peripheral system of classical links is a complete invariant \cite{Waldo}. 
Since this invariant extends to welded links, in the sense that the above-defined invariant coincides with the usual peripheral system for classical links, this shows that classical links inject into welded links \cite{Kauffman,GPV}.  
\end{remarque}

\subsection{Reduced group and reduced peripheral system}

As before, let us consider a welded diagram $L$. 

\begin{defi}
For a group $G$ given with a finite generating set $X$, the \emph{reduced group of $G$}, denoted by $\nR G$, is the quotient of $G$ 
by its normal subgroup generated by all elements $[\zeta,\omega^{-1}\zeta \omega]$, where $\zeta\in X$ and $\omega\in G$. 
In particular, we define the \emph{reduced group of $L$} as the reduced group $\nR G(L)$ of $G(L)$ with respect to its
Wirtinger generators. 
\end{defi}

Note that $\nR G(L)$ is the largest quotient of $G(L)$
where any meridian commutes with any of its conjugates. Since any two
meridians of a same component are conjugate elements, it can also be
defined as the quotient of $G(L)$ by the normal subgroup generated by
the elements $[\mu_i,\omega^{-1}\mu_i \omega]$ for all $\omega\in G(L)$, where $\{\mu_i\}_i$ is a fixed basing for $L$. 

\begin{convention*} 
In the rest of this paper, we shall use greek letters with tilda for elements in the group of a welded diagram, 
and use the same letters, but without the tilda, to denote the corresponding elements in the reduced group. 
In particular, we respectively denote by $\mu_i$ and $\lambda_i$ the images in $\nR G(L)$ of any meridian  $\widetilde{\mu}_i $ and  longitude $\widetilde{\lambda}_i$ in $G(L)$.
\end{convention*}

\begin{defi}
The \emph{reduced peripheral system} for $L$ is the data 
$$\big( \nR G(L),\{(\mu_i,\lambda_i.N_i)\}_i\big),$$ 
associated to a peripheral system $\big(G(L),\{(\widetilde{\mu}_i,\widetilde{\lambda}_i)\}_i\big)$
where, for each $i$,  $\lambda_i.N_i$ denotes the coset of $\lambda_i$ with
respect to $N_i$, the normal subgroup generated by the $i\,$th reduced
meridian $\mu_i$. 
Two reduced peripheral systems are \emph{conjugate} 
if they come from conjugate peripheral systems; and they are 
\emph{equivalent} if there is a group isomorphism sending one 
to a conjugate of the other.
\end{defi}

As explained in the introduction, Milnor introduced the reduced peripheral system 
for classical links, and showed that it is a link-homotopy invariant. We have the following generalization. 
\begin{lemme}\label{lem:lhinv}
Up to equivalence, the reduced peripheral system is a well-defined
  invariant of welded links up to $\sv$--equivalence. 
\end{lemme}
\begin{proof} 
Since equivalent peripheral systems obviously yield equivalent
  reduced peripheral systems, it suffices to prove the invariance under a single SV move. 
Pick a self-crossing $s$ of some welded diagram $L$, and denote by
$L_s$ the diagram obtained by replacing $s$ by a virtual crossing:
   \[
   \begin{array}{c}
L\ \dessin{1.4cm}{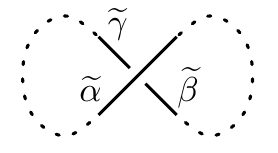} \hspace{.3cm} \stackrel{\textrm{SV}}{\longleftrightarrow} \hspace{.3cm} \dessin{1.4cm}{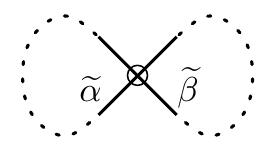}\ L_s 
    \end{array}.
  \]
Consider the three generators $\widetilde{\alpha},\widetilde{\beta},\widetilde{\gamma}$ of $G(L)$ involved in $s$, as
shown above.
Since the meridians $\widetilde{\alpha}$, $\widetilde{\beta}$ and $\widetilde{\gamma}$ all belong to the same
  component, there are $\widetilde{\omega},\widetilde{\zeta}\in G(L)$ such that
  $\widetilde{\beta}=\widetilde{\omega}^{-1}\widetilde{\alpha}\widetilde{\omega}$ and $\widetilde{\alpha}=\widetilde{\zeta}^{-1}\widetilde{\gamma}\widetilde{\zeta}$.
For $L$, the Wirtinger relation at $s$ is $\widetilde{\gamma}=\widetilde{\alpha}^{-1}\widetilde{\beta}\widetilde{\alpha}$; hence, we have that $\gamma=\alpha^{-1}\omega^{-1}\alpha \omega \alpha\equiv \omega^{-1}\alpha \omega=\beta$ holds in $\nR G(L)$, which shows that $\nR G(L)$ is isomorphic to $\nR G(L_s)$. 

It remains to show that this isomorphism preserves the reduced peripheral system. 
Pick $\widetilde{\alpha}\in G(L)$ as meridian $\widetilde{\mu}$ for the component of $L$ containing $s$; the corresponding preferred longitude is given by 
$\widetilde{\lambda}=\widetilde{\omega} \widetilde{\alpha}\widetilde{\zeta} \widetilde{\alpha}^{-k}$ for some integer $k$. Take the meridian $\widetilde{\mu}_s$ of the corresponding component of $L_s$ to be represented by $\widetilde{\alpha}$ again, so that the 
preferred longitude is given by $\widetilde{\lambda}_s=\widetilde{\omega}\widetilde{\zeta} \widetilde{ \alpha}^{-k+1}$. 
Then the isomorphism $\nR G(L)\rightarrow \nR G(L_s)$ maps $\mu$ to $\mu_s$, and the equality
\[
\lambda = \omega \alpha\zeta \alpha^{-k} \equiv \omega[\zeta,\alpha]\alpha\zeta \alpha^{-k} =
\omega\zeta \alpha\zeta^{-1}\alpha^{-1}\alpha\zeta \alpha^{-k}=\omega\zeta \alpha^{-k+1} \textrm{ (mod $N$)}
\]
shows that $\lambda.N$ is mapped to $\lambda_s.N_s$, where $N\subset
  \nR G(L)$ and $N_s\subset
  \nR G(L_s)$ denote the normal subgroups generated by
$\alpha$. 
This handles one version of the SV move, but the other one is strictly similar.
\end{proof}
\begin{remarque}
Since we are overall working modulo the normal subgroup generated by the meridian $\alpha$, the above sequence of equalities for the longitude $\lambda$ could be slightly simplified. It seems however instructive to point out that one only needs to consider this normal subgroup in the second equality of this sequence;
this identifies precisely where  this equivalence is needed to have the desired invariance property. 
\end{remarque}

In particular, the reduced fundamental group is hence invariant under self-virtualization. Combining this with the sorted form given in Lemma \ref{lemma:Sorted}, 
we obtain the following presentation. 

\begin{lemme}\label{lem:Rpres}
The reduced fundamental group $\nR G(L)$ has the following presentation 
 \[
\nR G(L) = \big\langle \mu_1,\ldots, \mu_n\,\,\big| \,\,  [\mu_i,\lambda_i], \,
[\mu_i,\omega^{-1}\mu_i \omega],  \textrm{ for all $i$ and for all $\omega\in  F(\mu_i)$}
\big\rangle,
\]
 where $F(\mu_i)$ denotes the free group on $\{\mu_i\}_i$. 
\end{lemme}

\begin{proof}
  Suppose first that $L$ corresponds to a sorted Gauss diagram. 
  The Wirtinger-type presentation for its welded link group provides, for every component $C_i$, 
  a generator $\mu_i$ corresponding to the $t$--arc, and a bunch of generators $\mu_i^j$ lying on the $h$--arc. 
  Any of the latter appears in exactly two relations which are of the form $\mu_i^j=\mu^{\pm1}_{i_1}\mu_i^{j-1}\mu^{\mp1}_{i_1}$ 
  and $\mu_i^{j+1}=\mu^{\pm1}_{i_2}\mu_i^j\mu^{\mp1}_{i_2}$, setting $\mu_i^0:=\mu_i=:\mu_i^{r_i}$ where $r_i$ is the number of 
  heads on $C_i$; generators $\mu_i^j$ can hence be successively eliminated, ending with relations $ [\mu_i,\lambda_i]$ only. 
  Hence the presentation  
  \[
   G(L) = \big\langle \mu_1,\ldots, \mu_n\,\,\big| \,\,  [\mu_i,\lambda_i], \, \textrm{ for all $i$ } \big\rangle.
  \]
  Relations $[\mu_i,\omega^{-1}\mu_i \omega]$ then naturally arise when taking the reduced quotient. 
  The general case, where $L$ does not necessarily correspond to a sorted Gauss diagram, then follows readily from Lemmas \ref{lemma:Sorted} and \ref{lem:lhinv}.    
\end{proof}

\begin{remarque}\label{rem:Presentation->Diagram}
  Starting with a sorted Gauss diagram $D$, Lemma \ref{lem:Rpres}
  provides a presentation for the associated reduced fundamental group
  whose generators $\{\mu_i\}_i$ correspond to the $t$--arcs, and whose relations
  make these generators commute with their own conjugates, as well as
  with their associated longitude. Moreover, as words in the
  generators, these longitudes are directly given by the words
  associated to $D$ in Remark \ref{rem:SortedLabels}. Conversely, any
  such presentation provides a word representative for each longitude
  and hence describes, up to OC moves, a unique sorted Gauss diagram.
\end{remarque}

\begin{remarque}
  Lemma \ref{lem:Rpres} is to be be compared with \cite[Thm.~4]{Milnor2}, where Milnor gives a similar presentation for the nilpotent quotients of the group of a classical link. 
  Actually, Milnor's argument being purely algebraic, the proof of \cite[Thm.~4]{Milnor2} applies verbatim to the case of welded links, 
  meaning in particular that the nilpotent quotient $G(L) / \Gamma_k G(L)$ has a presentation 
  $\big\langle \mu_1,\ldots, \mu_n\,\,\big| \,\,  \Gamma_k F(\mu_i),\,  [\mu_i,\lambda_i]  \textrm{ for all $i$}\big\rangle$. 
  In fact, Milnor's proof can be adapted with minor adjustments to give an alternative proof of Lemma \ref{lem:Rpres}.
\end{remarque}

\section{Diagrammatic characterization of the reduced peripheral system}\label{sec:2}

This section is devoted to the proof of the following result, which readily implies the equivalence $\textrm{i}\Leftrightarrow \textrm{ii}$ in our main theorem, and which more generally classifies welded links up to $\sv$--equivalence. 

\begin{theo}\label{thm:weldedclassif}
 Two welded links are $\sv$--equivalent if and only if they have isomorphic reduced peripheral systems.
\end{theo}

For convenience, we will adopt here the Gauss diagram point of view,
fix a number $n\in\N^*$ of components and set $\nF(\mu)$ the free group
over elements $\mu_1,\ldots,\mu_n$. 

\begin{proof}
 The ``only if'' part of Theorem \ref{thm:weldedclassif} was proved in Lemma \ref{lem:lhinv}.
Conversely, let $L$ and $L'$ be two welded diagrams
 with equivalent reduced peripheral systems $\big( \nR
 G(L),\{(\mu_i,\lambda_i.N_i)\}_i\big)$ and $\big( \nR
 G(L'),\{(\mu'_i,\lambda'_i.N'_i)\}_i\big)$. Using Lemma
 \ref{lemma:Sorted}, we may assume that $L$ and $L'$ are both given by
 sorted Gauss diagrams $D$ and $D'$. Following Remark
 \ref{rem:Presentation->Diagram}, the strategy will be to apply welded
 and $\SV$ moves on $D'$ so that the presentations for the reduced fundamental group associated to $D$ and $D'$ are the same, meaning that $D$ and $D'$ are the same up to OC moves. To keep notation light, we
 will also denote by $D'$ all the sorted Gauss diagrams successively obtained by
 modifying $D'$. 

\begin{figure}
  \[
  \begin{array}{cc}
\dessin{2cm}{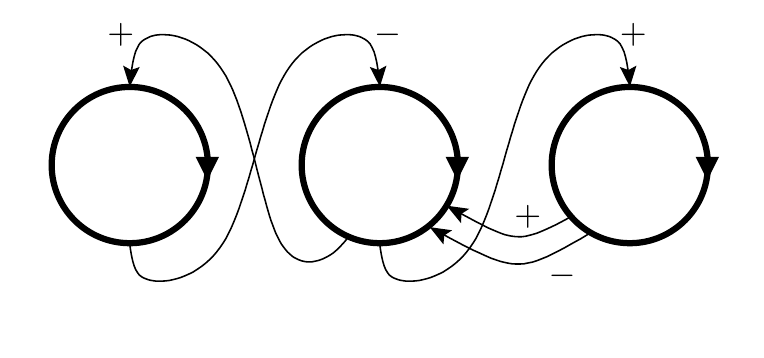}& 
\xrightarrow[]{\textrm{TaH}}\ \dessin{2cm}{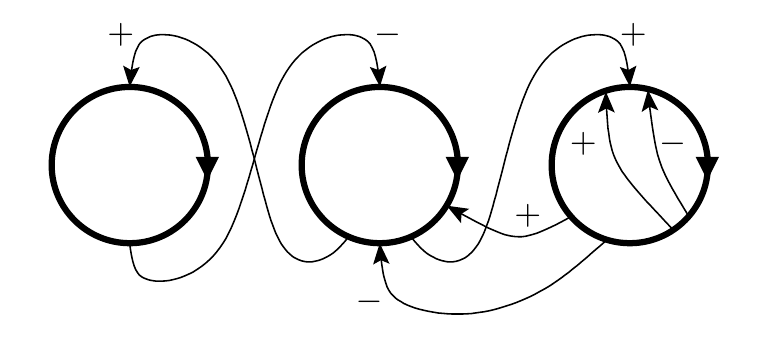}\ \xrightarrow[]{\textrm{TaH}}\ \dessin{2cm}{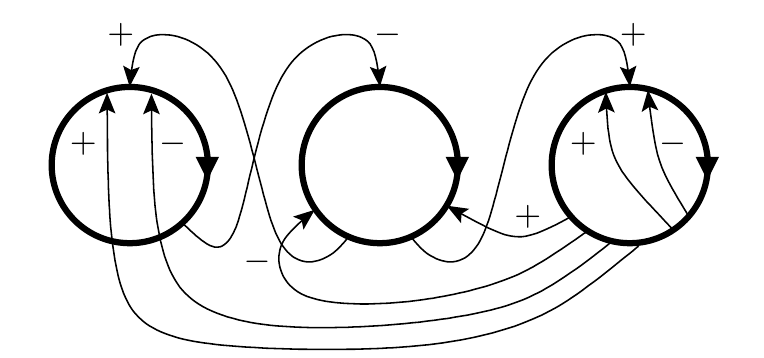}
  \\[1cm]
\phantom{\textrm{\scriptsize R2}}\uparrow \textrm{\scriptsize R2}&  \\[.1cm]
\dessin{2cm}{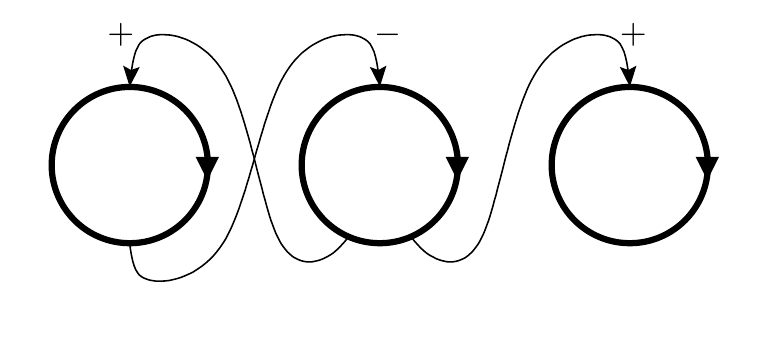}&            \\[-.2cm]
\phantom{\textrm{\scriptsize R2}}\downarrow \textrm{\scriptsize R2}&\\[.1cm]
\dessin{2cm}{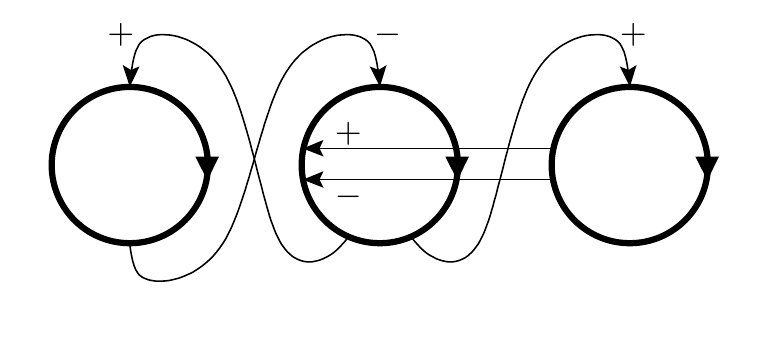} & 
\xrightarrow[]{\textrm{R2}}\ \dessin{2cm}{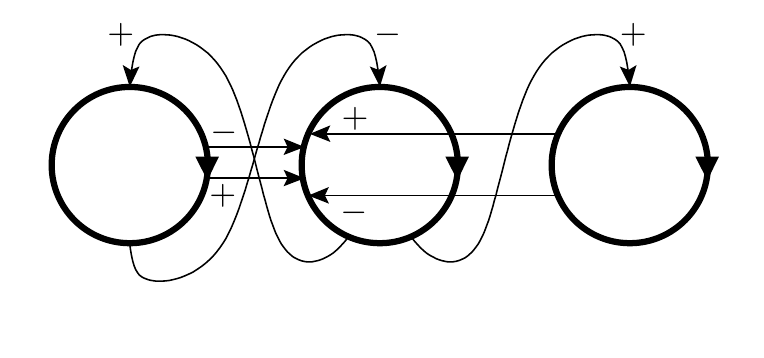}\ \xrightarrow[]{\textrm{SV}}\ \dessin{2cm}{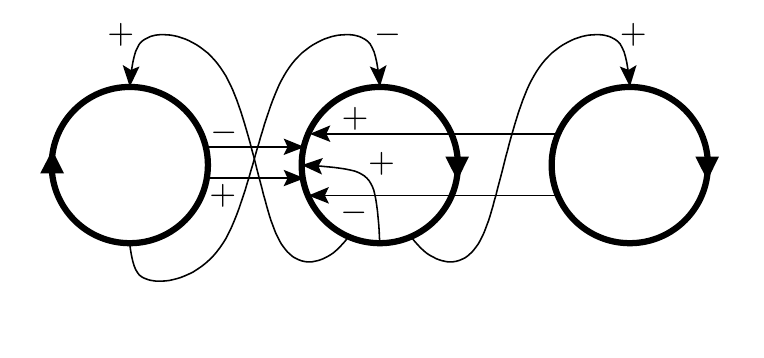}
\\[-.1cm]
  \end{array}
\]
  \caption{Making the equivalence an isomorphism}
  \label{fig:Equivalence->Isomorphism}
\end{figure}

To begin, we first modify $D'$ so that the
isomorphism $\psi:\nR G(L')\to\nR G(L)$ sends each pair 
$(\mu'_i,\lambda'_i)$ to $(\mu_i,\lambda_i)$. 
Algebraically, transforming $\psi(\mu'_i,\lambda'_i)$ into $(\mu_i,\lambda_i)$ can be achieved in two steps: 
\begin{enumerate}
\item since $\psi(\mu'_i)$ is a conjugate of $\mu_i$, 
  perform a sequence of ``elementary conjugations'', each replacing both $\mu'_i$ and $\lambda'_i$ by their conjugate under
${\mu_j'}^{\e}$, for some index $j\neq i$ and sign $\e$;
\item since $\psi(\lambda'_i)$ is an
  element of $\lambda_i.N_i$, multiply $\lambda'_i$ by 
  an appropriate product of conjugates of ${\mu'_i}^{\pm1}$.
\end{enumerate}
These two steps have to be realized at the diagrammatical level:
\begin{enumerate}
\item  for each
elementary conjugation, use R2 to add two parallel arrows going from the $j\,$th $t$--arc to the 
starting extremity of the $i\,$th $t$--arc, and then pull the head of the
$(-\e)$--labelled arrow along the $i\,$th $t$--arc using the $\TaH$ move given in Figure \ref{fig:TaH}. 
This yields an equivalent Gauss diagram which is still sorted, and
which realizes the desired elementary conjugation. See the upper half of Figure \ref{fig:Equivalence->Isomorphism} for an example where $\mu'_2$ and $\lambda'_2$ are conjugated by $\mu'_3$, the components being numbered from left to right;
\item for adding a conjugate $g{\mu'_i}^{\pm1}g^{-1}$ of ${\mu'_i}^{\pm1}$ to $\lambda'_i$,
use repeatedly R2 to add the trivial word $gg^{-1}$ to the $i$th
longitude of $D'$, and then a single SV move to produce the desired word $g{\mu'_i}^{\pm1}g^{-1}$. See the lower half of Figure \ref{fig:Equivalence->Isomorphism} for an example where the conjugate of $\mu'_2$ by ${\mu_1'}^{-1}\mu_3'$ is introduced in $\lambda'_2$, the components being numbered from left to right. 
\end{enumerate}

We have thus realized the isomorphism $\psi:\nR G(L')\to\nR G(L)$ which sends each $\mu'_i$ to $\mu_i$, and each $\lambda'_i$ to $\lambda_i$. 
As a matter of fact, we can now identify $\mu'_i$ with $\mu_i$, and according to the presentation given in Lemma \ref{lem:Rpres}, $\psi(\lambda'_i)$ and $\lambda_i$ differ, as words in the $\mu_j$'s, by a sequence of the following moves:
  \begin{enumerate}
  \item[i.] $\mu_j^{\pm1}\mu_j^{\mp1}\leftrightarrow 1$;
  \item[ii.] $\omega\mu^{\pm1}_j\rightarrow \mu^{\pm1}_j\omega$ where $\omega$ is
    of the form $\zeta^{-1}\mu^{\pm1}_j\zeta$ with $\zeta\in \nF(\mu)$;
  \item[iii.] $\mu^{\pm1}_j\rightarrow \omega^{-1}\mu^{\pm1}_j\omega$ where
    $\omega$ is any representative in $\nF(\mu)$ of the element $\lambda_j$.
  \end{enumerate}
But each of these moves can be realized by modifying $D'$.  Relations
i correspond indeed to R2 moves. 
Relations ii can be handled exactly as in the proof of \cite[Lem. 4.26]{ABMW}. 
For relations iii, consider first the particular representative
$\omega_0$ of $\lambda_j$ given by $D'$ following Remark \ref{rem:SortedLabels}. 
At the level of $D'$, 
the term $\mu^{\pm1}_j$ in relation iii corresponds to an arrow $a$
whose tail sits on the $j\,$th circle; moving this tail along the
whole circle component, against the orientation, 
does conjugate $\mu^{\pm1}_j$ by $\omega_0$. Indeed, using
$\TaH$ moves, the tail of $a$ will cross
every head on its way at the cost of conjugating the head of $a$ with
the desired arrows; see Figure \ref{fig:ConjugatingMu} for an example where the $\mu_2$ factor in $\lambda_1'$ is conjugated by $\lambda_2'$, the components being numbered from left to right. For other representatives of $\lambda_j$, we note that
they differ from $\omega_0$ by a sequence of the following moves: 
\begin{itemize}
\item[i'.] $\mu_k^{\pm1}\mu_k^{\mp1}\leftrightarrow 1$;
\item[ii'.] $\zeta^{-1}\mu^{\pm1}_k\zeta\rightarrow \mu^{\pm1}_k$ where
    $\zeta$ is any element in
    $\nF(\mu)$.
\end{itemize}
Before sliding the tail of $a$ over the circle component, $D'$ should hence be modified so
that the slide operation does conjugate by the right word in
$\nF(\mu)$. Again, relations i' can be realized using R2 moves. For relation ii', use the Slide move to remove by pairs the
arrows of $\zeta$ and $\zeta^{-1}$ away from the circle. 
Then move the tail of $a$ along this loop, and perform the relations i' and ii' backwards. 
\begin{figure}
  \[
\dessin{2cm}{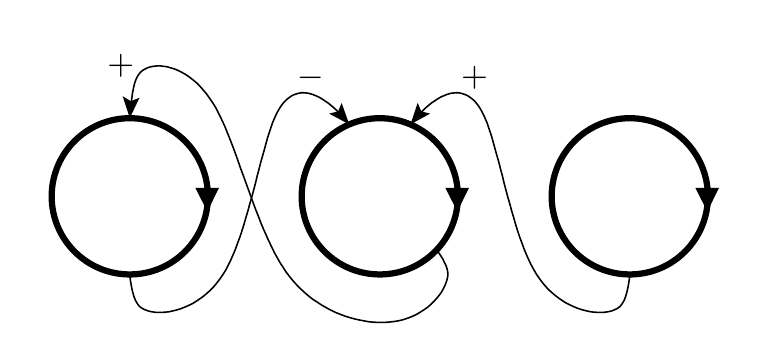}\ \xrightarrow[]{\textrm{TaH}}\ \dessin{2cm}{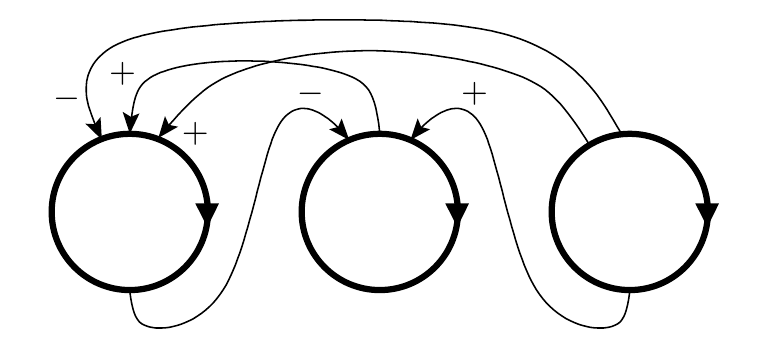}\ \xrightarrow[]{\textrm{TaH}}\ \dessin{2cm}{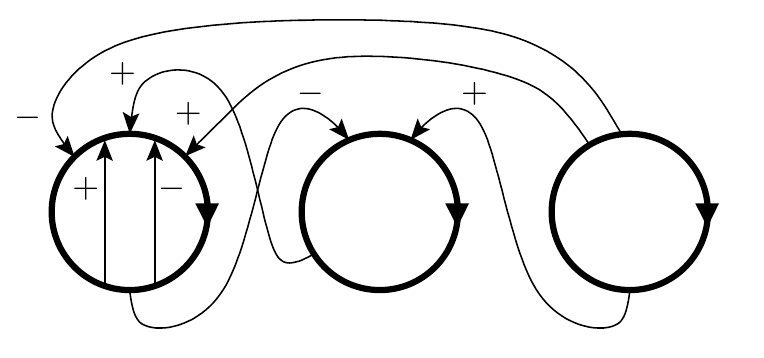}
\]
  \caption{Conjugating meridian factors in longitudes}
  \label{fig:ConjugatingMu}
\end{figure}
\end{proof}

\section{A topological characterization of the reduced peripheral system}\label{sec:3}

Welded links are closely related to \emph{ribbon knotted tori}  and \emph{ribbon solid tori} in
$S^4$, and the characterization of classical links having same
  reduced peripheral systems given by Theorem \ref{thm:weldedclassif} can be recasted in terms of
  $4$--dimensional topology.

\subsection{The enhanced Spun map} 

Given a classical link $L\subset\R^3$, a well-known procedure to
construct ribbon knotted tori in $4$--space is to take the \emph{Spun} of $L$:
consider a plane $\PP$ which is disjoint from a $3$--ball containing
$L$, and spin $L$ around $\PP$ inside $\R^4\supset\R^3$. The 
  result is a union of knotted tori, which we denote by $\Spun(L)$. If the
  projection $D(L,\PP)$ of $L$ onto the plane $\PP$ is regular, then
  spinning as well the orthogonal projection rays from $L$ to $\PP$
  provides immersed solid tori whose boundary is $\Spun(L)$ and whose singularities are so-called
  \emph{ribbon} disks, corresponding to
  the crossings of $D(L,\PP)$. 
Of course, this \emph{ribbon filling} depends on the choice of plane
$\PP$, and more precisely on the diagram $D(L,\PP)$, which may by
changed by some sequence of Reidemeister moves.
But for each Reidemeister move, there is an associated singular diagram, that is a singular plane $\PP_s$, 
and spinning $L$ around $\PP_s$ provides some singular ribbon filling which can be infinitesimally desingularized into the spun of one or the other side of the Reidemeister move. 
This leads to the following definition, which settles a notion of
(singular) ribbon solid tori. 
\begin{defi}
  Let $\varphi:M\to S^4$ be an immersed $3$--dimensional manifold.
  Let $D$ be a connected component of the singular set of $\varphi(M)$
  contained in an open $4$--ball $B\subset S^4$. 
  We say that $D$ is a \emph{ribbon singularity}:\footnote{For the sake of exactitude, we provide here explicit formulas for the singularities, but informal descriptions follow in Remark \ref{rem:InformalRibbonSingularities}.} 
  \begin{itemize}
  \item \emph{of type $0$} if $\varphi^{-1}(B)$ is the disjoint union  $B_1\sqcup B_2$ of
    two $3$--balls and there is a local system of coordinates for
    $B\cong\R^4$ such that $\left\{
  \begin{array}{l}
    \varphi(B_1)=\Big\{\big(t,r\cos(s),r\sin(s),0\big)\ \big|\
    t,s\in\R,r\in[0,2]\Big\}\\[.2cm]
    \varphi(B_2)=\Big\{\big(0,r\cos(s),r\sin(s),t\big)\ \big|\
    t,s\in\R,r\in[0,1]\Big\}
  \end{array}
\right.$;
  \item \emph{of type $2$} if $\varphi^{-1}(B)$ is the disjoint union $B_1\sqcup B_2$ of
    two $3$--balls and there is a local system of coordinates for
    $B\cong\R^4$ such that $\left\{
  \begin{array}{l}
    \varphi(B_1)=\Big\{\big(t,r\cos(s),r\sin(s),t^2\big)\ \big|\
    t,s\in\R,r\in[0,2]\Big\}\\[.2cm]
    \varphi(B_2)=\Big\{\big(t,r\cos(s),r\sin(s),-t^2\big)\ \big|\
    t,s\in\R,r\in[0,1]\Big\}
  \end{array}
\right.$;
  \item \emph{of type $3$} if $\varphi^{-1}(B)$ is the disjoint union  $B_1\sqcup B_2\sqcup
    B_3$ of
    three $3$--balls and there is a local system of coordinates for
    $B\cong\R^4$ such that $\left\{
  \begin{array}{l}
    \varphi(B_1)=\Big\{\big(t,r\cos(s),r\sin(s),0\big)\ \big|\
    t,s\in\R,r\in[0,2]\Big\}\\[.2cm]
    \varphi(B_2)=\Big\{\big(t,r\cos(s),r\sin(s),t\big)\ \big|\
    t,s\in\R,r\in[0,1]\Big\}\\[.2cm]
    \varphi(B_3)=\Big\{\big(t,r\cos(s),r\sin(s),-t\big)\ \big|\
    t,s\in\R,r\in\big[0,\frac12\big]\Big\}
  \end{array}
\right.$;
  \item \emph{of type $\SV$} if $\varphi^{-1}(B)$ is the disjoint
    union  $B_1\sqcup B_2$ of two $3$--balls, $B_1$ and $B_2$ belongs
    to the same connected component of $M$, and there is a local system
    of coordinates for $B\cong\R^4$ such that $\left\{
  \begin{array}{l}
    \varphi(B_1)=\Big\{\big(r,t,s,0\big)\ \big|\
    t,s\in\R,r\in\R_-\Big\}\\[.2cm]
    \varphi(B_2)=\Big\{\big(0,r\cos(s),r\sin(s),t\big)\ \big|\
    t,s\in\R,r\in[0,1],r\Big\}
  \end{array}
\right.$.
  \end{itemize}
\end{defi}

\begin{remarque}\label{rem:InformalRibbonSingularities}
  In all four cases, the ribbon singularity $D$ corresponds to the disk $\Big\{\big(0,r\cos(s),r\sin(s),0\big)\ \big|\
    s\in\R,r\in[0,1]\Big\}$.
Type $0$ corresponds to two solid
    tubes, one being smaller than the other, intersecting
    transversally; these are the usual ribbon singularities. Type $2$
    corresponds to two solid tubes, one being smaller than the other,
    intersecting tangentially; these occur when spinning a link
    around a plane on which the link projects with two tangential
    strands. Type $3$
    corresponds to three solid tubes of increasing width,
    intersecting simultaneously and transversally; these occur when spinning a link
    around a plane on which the link projects with a triple
    point. Type $\SV$ differs from type $0$ in 
    that one
    preimage of the singular disk
    lies on the boundary of $M$
    instead of its interior, and in that the two preimages belong to the
    same solid torus; 
    these occur when performing the link-homotopy which pushes at once a usual
    ribbon (self) singularity through the boundary of $M$. Note that 
    a type $1$ seems to be missing here, which would correspond to spinning a link
    around a plane on which the link projects with a cusp, but this
    does not introduce any new kind of ribbon singularity.
\end{remarque}

\begin{defi}
  \emph{Ribbon solid tori} are immersed solid tori in $S^4$ whose
  singular locus is made of ribbon singularities of type
  $0$. \emph{Generalized ribbon solid tori} are immersed solid tori in $S^4$ whose
  singular locus is made of ribbon singularities of type $0$, $2$ and
  $3$. \emph{Self-singular ribbon solid tori} are immersed solid tori in
  $S^4$ whose singular locus is made of ribbon singularities of type
  $0$ and $\SV$. 

  We say that two (generalized) ribbon solid tori are \emph{equivalent} if there is
  a path among generalized ribbon solid tori connecting them, and we
  say that they are \emph{ribbon link-homotopic} if there is path
  among generalized and self-singular ribbon solid tori connecting them. 
\end{defi}

Adding the spun of projection rays in the above definition of the $\Spun$ map provides a well-defined map $\Spun^\bullet$ from classical links to generalized ribbon soli tori. 
The following result is the topological
characterization of the reduced peripheral system given by the equivalence $\textrm{i}\Leftrightarrow \textrm{iii}$ in our main theorem. 
\begin{theo}\label{tralalilala}
  Two classical links $L$ and $L'$ have isomorphic reduced peripheral
 systems if and only if $\Spun^\bullet(L)$ and $\Spun^\bullet(L')$ are
 ribbon link-homotopic.
\end{theo}
The proof is given in the next section. 
As for the diagrammatic characterization given in Section \ref{sec:2},
this will follow from a more general result, Theorem \ref{thm:4-dimclassif}, characterizing the reduced peripheral system of welded links in terms of $4$--dimensional topology. 

\subsection{The enhanced Tube map}

In this section, we prove Theorem \ref{tralalilala}, using the so-called $\Tube$ map. 
Recall from \cite{Satoh} that Satoh's generalization of Yajima's $\Tube$ map is defined from welded links to ribbon knotted $2$-tori, and that for any welded link $L$, $\Tube(L)$ actually comes with a canonical ribbon filling. 
In order to fully record this ribbon filling in the $\Tube$ map, and to connect with the $\Spun^\bullet$ map, we are led to the following notion.  
\begin{defi}
  We define \emph{generalized welded diagrams} as diagrams with cups and the following
  kind of crossings:
\[
\begin{array}{c}
  \dessin{1.2cm}{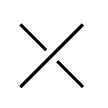}\\
  \textrm{classical}
\end{array}
\hspace{1cm}
\begin{array}{c}
  \dessin{1.2cm}{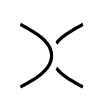}\\
  \textrm{tangential}
\end{array}
\hspace{1cm}
\begin{array}{c}
  \dessin{1.2cm}{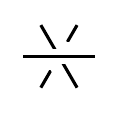}\\
  \textrm{triple}
\end{array}
\hspace{1cm}
\begin{array}{c}
  \dessin{1.2cm}{CVirtuel}\\
  \textrm{virtual}
\end{array}.
\]
Then classical Reidemeister moves are replaced by a path of diagrams
going through the corresponding cusp, tangential ou triple
point. Other welded moves are still locally allowed.

We define \emph{self-singular welded diagrams} as diagrams with the following
  kind of crossings, where the two strands involved in a semi-virtual crossing belong to a same component:
\[
\begin{array}{c}
  \dessin{1.2cm}{CClassical}\\
  \textrm{classical}
\end{array}
\hspace{1cm}
\begin{array}{c}
  \dessin{1.2cm}{CVirtuel}\\
  \textrm{virtual}
\end{array}
\hspace{1cm}
\begin{array}{c}
  \dessin{1.2cm}{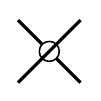}\\
  \textrm{semi-virtual\footnotemark}
\end{array}.
\]
\footnotetext{Semi-virtual crossings were already introduced in
  \cite{GPV} in connection with finite type invariants of virtual
  knots.}
\emph{Self-virtualization} is defined for generalized welded diagrams
as the equivalence relation generated by the local moves turning a
semi-virtual crossing into either a classical crossing, or a virtual
one as follows:
\[
\dessin{1.2cm}{CClassical} \longleftrightarrow
\dessin{1.2cm}{CSemiVirtuel}
\longleftrightarrow  \dessin{1.2cm}{CVirtuel}.
\]
\end{defi}
Following \cite[Sec. 3.2]{Aud_HdR}, one can then define a map
\[
\Tube^\bullet:\frac{\big\{\textrm{generalized welded
    diagrams}\big\}}{\textrm{self-virtualization}}\to\frac{\big\{\textrm{generalized
  ribbon solid tori}\big\}}{\textrm{ribbon
  link-homotopy}}
\]
which, respectively, associates ribbon singularities of type $0$, $2$,
$3$ and $\SV$ to classical, tangential, triple and semi-virtual crossings 
and connects these various singularities by pairwise disjoint $3$-balls, as prescribed by the welded diagram. 
It is then a straightforward
adaptation of \cite[Prop. 3.7]{Aud_HdR} to prove that $\Tube^\bullet$
is one-to-one.

As a direct corollary of Theorem \ref{thm:weldedclassif}, we obtain the following alternative characterization of the reduced peripheral system, which holds for all welded links. 
\begin{theo}\label{thm:4-dimclassif}
 Two welded links $L_1$ and $L_2$ have isomorphic reduced peripheral
 systems if and only if $\Tube^\bullet(L_1)$ and $\Tube^\bullet(L_2)$ are
 ribbon link-homotopic.
\end{theo}

Theorem \ref{tralalilala} follows from this results. 
Indeed, as essentially pointed out by Satoh in \cite{Satoh}, it is
clear from the definition of the $\Spun^\bullet$ map that, starting
with a diagram $D$ of a classical link $L$, the ribbon solid tori
$\Spun^\bullet(L)$ consists of ribbon singularities which are
connected by $3$-balls as combinatorially prescribed by $D$:
$\Spun^\bullet(L)$ and $\Tube^\bullet(L)$ are hence equivalent. 

\subsection{Link-homotopy of ribbon surfaces in $4$-space}\label{sec:surfaces}

The original versions of the $\Spun$ and $\Tube$ maps produce ribbon $2$--tori, 
which are just the boundary of some ribbon solid tori, rather than $3$-dimensional objects.  
Obviously, any ribbon link-homotopy between two ribbon solid tori
induces a usual link-homotopy between their 
boundaries. Building on this remark, it follows that:
\begin{prop}
  If two classical links $L_1$ and $L_2$ have isomorphic reduced
  peripheral systems, then $\Spun(L_1)$ and $\Spun(L_2)$ are link-homotopic.
\end{prop}
It is hence tempting to hope for the converse to hold true: this would 
give a topological characterization of the reduced peripheral system in terms of spun surfaces up to link-homotopy. 
However, this is not the case. 
There is indeed a known global move on welded links, related to the torus eversion in $S^4$, under
which the $\Spun$ map is invariant, and this move
transforms every classical link into its \emph{reversed image}, which is the
mirror image with reversed orientation (see \cite{Blake} or
\cite[Prop. 2.7]{IndianPaper}). 
Furthermore, 
it can be checked
that the (reduced) peripheral system of a reversed image is given
from the initial one by just inverting the longitudes.
It follows easily
from these two observations
that, for instance, the positive and negative Hopf links have non-equivalent reduced peripheral systems,
  whereas their spuns are isotopic, hence link-homotopic.
As a consequence, keeping track of the ribbon filling is mandatory to
preserve (reduced) peripheral systems, and Theorem \ref{tralalilala} is in this sense optimal.

\appendix
\section{Hughes' counterexample}\label{onvamangerdeschips}

As mentioned in the introduction, the fact that the reduced peripheral
system of classical links is not a complete link-homotopy invariant
was made explicit by J.~Hughes in \cite{Hughes93}. 
There, a pair of $4$-component links is given, which have isomorphic reduced peripheral systems but are not link-homotopic; this latter fact is proved using Levine's refinement of Milnor invariants developed in \cite{Levine}. 
These two links, $H_1$ and $H_2$, are given by the closures of the
following pure braids, oriented from left to right:
  \begin{gather*}
H_1:\dessin{1.8cm}{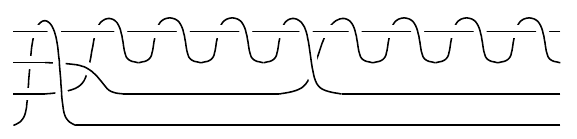};\\
H_2:\dessin{1.8cm}{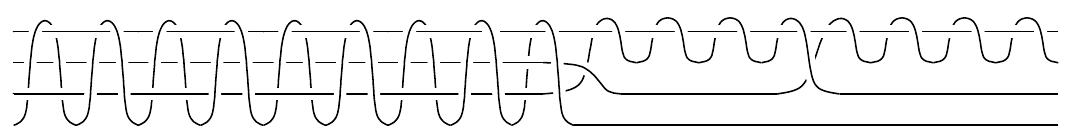}.
  \end{gather*}

Our main theorem implies that, although not link-homotopic, the links $H_1$ and $H_2$ are $\sv$--equivalent. 
This fact, however, is rather difficult to verify by hand, and we outline in this appendix the method that we used for this verification.  

We make use of the theory of \emph{arrow calculus} developed in \cite{arrow}, which is in some sense a `higher order Gauss diagram' theory. 
We only give here a quick overview of this theory, and refer to \cite{arrow} for precisions and further details.
Roughly speaking, a w-arrow 
for a diagram $L$ is an oriented
interval, possibly decorated by a dot, immersed in the plane so
that the endpoints lie on $L$; one can perform \emph{surgery} on $L$
along this w-arrow 
to obtain a new diagram as follows:
\[
\dessin{1.5cm}{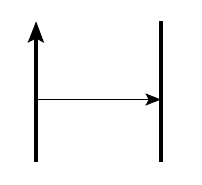}\leadsto\dessin{1.5cm}{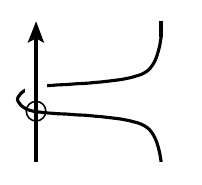}
\hspace{1cm}
\dessin{1.5cm}{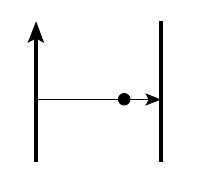}\leadsto\dessin{1.5cm}{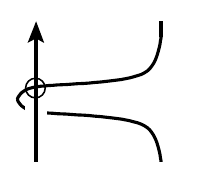}.
\]

More generally, one defines 
w-trees, which are oriented unitrivalent trees defined recursively by
the rules:
\[
\dessin{1.5cm}{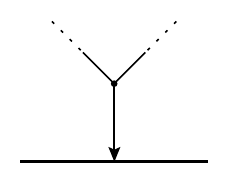}:=\dessin{1.5cm}{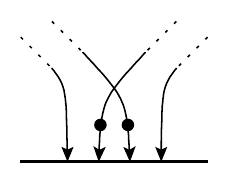}
\hspace{1cm}
\dessin{1.5cm}{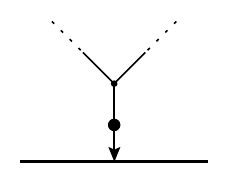}:=\dessin{1.5cm}{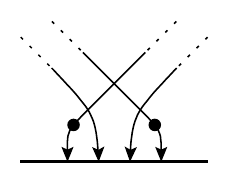}.
\]
There, the dotted parts represent `parallel' subtrees, see \cite[Conv.~5.1]{arrow}. 
Any welded diagram can be represented as a diagram without any crossing, but  with w-trees. 
For example, the links $H_1$ and $H_2$ can be represented in this way as the
closures of the diagrams given in Figure \ref{fig:ugh2} (ignoring the integer labels).
Note that the presentation for $H_2$ only differs from that for $H_1$
by the addition of a union $Y$ of Y--shaped w-trees. 

\begin{figure}
  \begin{gather*}
    H_1:\dessin{2.2cm}{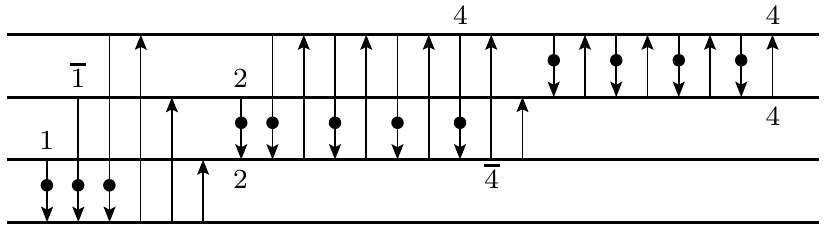}\\
    H_2:
    \makebox[0cm]{$\hspace{6.46cm}\underbrace{\raisebox{-.8cm}{\hspace{6.2cm}}}_{\hbox{$Y$}}$}
    \dessin{1.8cm}{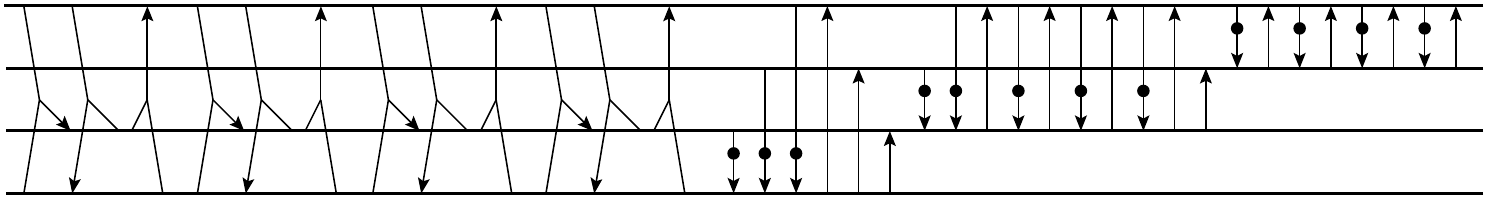}
  \end{gather*}
\caption{Arrow presentations of Hughes' links $H_1$ and $H_2$}
  \label{fig:ugh2}
\end{figure}

It is shown in \cite[Sections 4--5]{arrow} that two dots on a same edge do cancel,
and that the following moves can be performed on w-trees:
\[
\dessin{1.23cm}{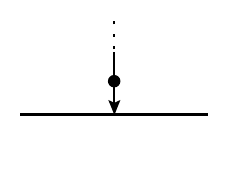}\leftrightarrow\dessin{1.23cm}{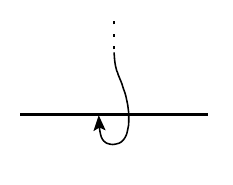}
\hspace{1cm}  
\dessin{1.23cm}{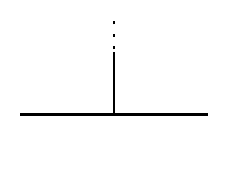}\leftrightarrow\dessin{1.23cm}{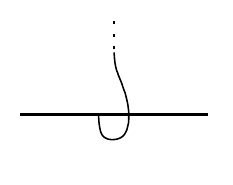}
\hspace{1cm}  
\dessin{1.36cm}{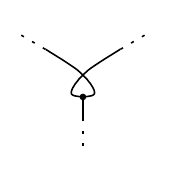}\leftrightarrow\dessin{1.36cm}{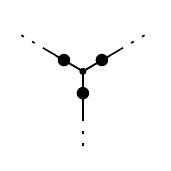}
\]
\[
\dessin{1.23cm}{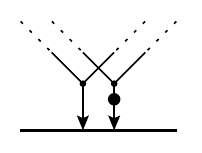}\leftrightarrow\dessin{1.23cm}{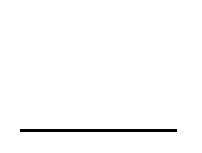}
\hspace{1cm}  
\dessin{1.23cm}{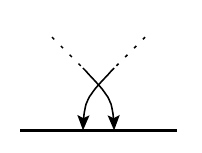}\leftrightarrow\dessin{1.23cm}{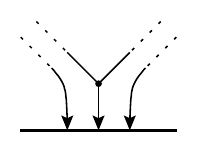}
\hspace{1cm}
\dessin{1.23cm}{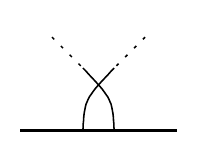}\leftrightarrow\dessin{1.23cm}{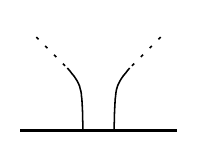}.
  \]
Moreover, up to $\sv$-equivalence, it is shown in \cite[Section 9]{arrow}  that 
so-called \emph{repeated} w-trees having at least two endpoints on a same connected component can be removed, 
and that the following moves can also be performed:
\[
\dessin{1.23cm}{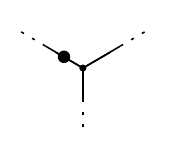}\leftrightarrow\dessin{1.23cm}{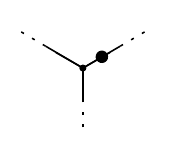}
\hspace{1cm}
  \dessin{1.23cm}{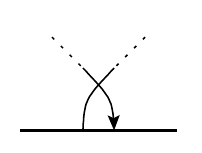}\leftrightarrow\dessin{1.23cm}{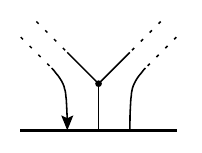}.
  \]
Then, one can start with the diagram for $H_1$ given in Figure \ref{fig:ugh2}, 
pick some w-arrow endpoint, and slide it all around the component it is attached to, in either direction: 
by the above moves, this will create $Y$-shaped w-trees, which can in turn be `gathered' at the cost of higher order w-trees. 
By performing the appropriate sequence of slides and cancelling inverse pairs and repeated w-trees, one can create
the union $Y$ of w-trees realizing the presentation for $H_2$,  
thus showing that the two links $H_1$ and $H_2$ are indeed $\sv$-equivalent. 
Such an appropriate sequence of slides is indicated in the upper part
of Figure \ref{fig:ugh2}: there, an integer label $k$ (resp. $\overline
k$) near an arrow
end indicates $k$ full turn in the left (resp. right)
direction.

\bibliographystyle{abbrv}
\bibliography{KnottedSurfaces}

\end{document}